\documentclass [11pt]{amsart}  
\usepackage{a4}                 
\usepackage{paralist}           
\usepackage{graphicx}
\usepackage{amssymb}                               
\usepackage{float}
\usepackage{amsmath}
\usepackage{psfrag}
\usepackage{MnSymbol}

\usepackage{amsmath,amsthm}

\theoremstyle{plain}
\newtheorem{thm}{Theorem}[section]

\newtheorem{lem}[thm]{Lemma}
\newtheorem{prop}[thm]{Proposition}
\newtheorem{cor}[thm]{Corollary}

\theoremstyle{definition}
\newtheorem{defn}[thm]{Definition}
\newtheorem{ques}{Question}

\theoremstyle{remark}

\newtheoremstyle{TheoremNum}
        {\topsep}{\topsep}              
        {\itshape}                      
        {}                              
        {\bfseries}                     
        {.}                             
        { }                             
        {\thmname{#1}\thmnote{ \bfseries #3}}
    \theoremstyle{TheoremNum}

\DeclareMathOperator{\image}{Im}

\DeclareMathOperator{\ab}{ab}
\DeclareMathOperator{\w}{w}
\DeclareMathOperator{\cond}{Cond}

\title{Finitely annihilated groups}
\author{Maurice Chiodo}
\date{\today}

\begin{document}

\begin{abstract}
In 1976 Wiegold asked if every finitely generated perfect group has weight $1$. We introduce a new property of groups, \emph{finitely annihilated}, and show this might be a possible approach to resolving Wiegold's problem. For finitely generated groups, we show that in several classes (finite, solvable, free), being finitely annihilated is equivalent to having non-cyclic abelianisation. However, we also construct an infinite family of (finitely presented) finitely annihilated groups with cyclic abelianisation. We apply our work to show that the weight of a non-perfect finite group, or a non-perfect finitely generated solvable group, is the same as the weight of its abelianisation. This recovers the known partial results on the Wiegold problem: a finite (or finitely generated solvable) perfect group has weight $1$.
\end{abstract}

\maketitle

\let\thefootnote\relax\footnotetext{2010 \textit{AMS Classification:} 20F99.}
\let\thefootnote\relax
\let\thefootnote\relax\footnotetext{\textit{The author was supported by:} an Australian Postgraduate Award, a David Hay Postgraduate Writing-Up Award, the Italian FIRB ``Futuro in Ricerca'' project RBFR10DGUA\_002, and the Swiss National Science Foundation grant FN PP00P2-144681/1.}

\section{Introduction}

The \emph{weight} of a group $G$, denoted $\w(G)$, is the smallest integer $n$ such that $G$ is the normal closure of $n$ elements. In 1976 Wiegold \cite[Question 5.52]{Notebook} posed the following problem, connecting groups of weight $1$ with perfect groups (those with trivial abelianisation):
\[
\textnormal{\emph{Does every finitely generated perfect group have weight $1$?}}
\]

There has been little progress on this problem since it was first posed. In this paper we try a new approach, by introducing a new property of groups (called \emph{finitely annihilated}), and connecting the problem with this new property.

Firstly, recall that a group is said to be \textit{residually finite} if every non-trivial element lies outside some (proper, normal) finite index subgroup. That is, the intersection of all proper, normal, finite index subgroups is the trivial element. Residually finite groups have been the subject of extensive study. They contain the class of fundamental groups of 3-manifolds, shown by combining a result of Hempel \cite{Hemp} with Perelman's solution to the Geometrization Conjecture \cite{Perel}. In the case of finitely presented groups, they have solvable word problem \cite{Mille-92}, and are Hopfian \cite{Har}. However, what if we were to invert this definition, and consider what would happen if we insisted that each element lies \textit{inside} (rather than outside) some proper normal finite index subgroup? This forms the basis for our new group property.

We say a group is \textit{finitely annihilated} (abbreviated to F-A) if it is the set-theoretic union of all its proper, normal, finite index subgroups. We use the term finitely annihilated because the property is equivalent to the following: $G$ is F-A if for every element $g\in G$ there is a finite non-trivial group $H$ and a surjection $\phi: G \twoheadrightarrow H$ with $\phi(g)=e$, i.e., each element is annihilated in some non-trivial finite quotient. Clearly, then, F-A groups must have (finite index) normal subgroups, and so cannot be simple.

Being F-A is independent of many well-studied group properties; that is, having one of these properties neither implies nor precludes being F-A. Straightforward examples of such properties include finite, solvable, hopfian, abelian, hyperbolic, free, and solvable word problem. Such examples come about from the fact that, if $H$ is an F-A group, and $G$ surjects onto $H$, then $G$ is also F-A (proposition \ref{quot}). Thus every group $G$ embeds into some F-A group (for example, $G \times C_{2} \times C_{2}$), so showing that a group $G$ is F-A can be reduced to showing that $G$ has some F-A quotient. Conversely, using the fact that every finitely generated group embeds into a finitely generated group with no finite index subgroups (lemma \ref{emb no fi}), we see that every finitely generated group embeds into a non F-A group. So F-A is a property that is far from being preserved by subgroups. What we can deduce, however, is that finitely generated groups which are not F-A must have cyclic (or trivial) abelianisation, as all finitely generated non-cyclic abelian groups are F-A (lemma \ref{abelian chi}).

An algebraic property $\rho$ of finitely presented groups is said to be \emph{Markov} if there exist two finitely presented groups $G_{+}, G_{-}$ such that $G_{+}$ has $\rho$, yet $G_{-}$ does not embed in \emph{any} finitely presented group with $\rho$. A property $\rho$ is \emph{co-Markov} if its complement (i.e., the property `not $\rho$') is Markov. The preceeding paragraph implies that F-A is neither a Markov property, nor a co-Markov property (corollary \ref{mar}).  A standard technique used to show that a group property is undecidable is to show that it is either Markov or co-Markov \cite{Mille-92}; such a strategy will not work for F-A groups. Whether being F-A is a decidable property amongst finitely presented groups remains an open problem. We note that a recent result by Bridson and Wilton \cite{BriWil} shows that having a non-trivial finite quotient is not a decidable property amongst finitely presented groups.

A useful way to show that a group  $G$ is \emph{not} F-A is to show that $G$ is the normal closure of a single element (that is, $\w(G)=1$). The converse does not hold though; in theorem \ref{mathieu} we construct a 3-generator finitely presented group which is neither F-A nor weight $1$. So F-A groups are not just the groups of weight greater than $1$. 

In our main technical result (theorem \ref{list}) we prove that, if a finitely generated group $G$ is free or sovable or finite, then $G$ is F-A if and only if $G^{\ab}$ (the abelianisation of $G$) is non-cyclic. This result enables us to recover the only known partial results on the Wiegold problem: every finite (or finitely generated solvable) perfect group has weight $1$ (corollary \ref{Wiegold fin}). This leads us to believe that understanding F-A groups could eventually prove very useful in resolving the Wiegold problem.

It would be tempting to try and show that \emph{all} finitely generated groups satisfy the conclusions of theorem \ref{list} (that is, a finitely generated group $G$ is F-A if and only if $G^{\ab}$ is non-cyclic). However, we use a construction by Howie \cite{Howie} to show that for any triple of distinct primes $p,q,r$, the group  $C_{p}*C_{q}*C_{r}$ is F-A and yet has cyclic (non-trivial) abelianisation (theorem \ref{non-compact}). It is an open questions as to whether there exists a finitely generated F-A group with trivial abelianisation (that is, a finitely generated perfect F-A group). If such a group exists then it must have weight greater than $1$, and so finding such a group would resolve the Wiegold problem in the negative.
\\ 

\noindent \textbf{Acknowledgements}: The author wishes to thank the graduate participants of the conference `Geometric Group Theory' (Poznan, June 2009) for their initial interest in this question, Rishi Vyas for his many long and thoughtful discussions on the results contained in this paper, and Jack Button for sharing his extensive knowledge of existing results in group theory.

\section{Definitions}

\subsection{Notation}

If $P=\langle X|R\rangle$ is a group presentation with generating set $X$ and relators $R$, then we denote by $\overline{P}$ the group presented by $P$; $P$ is said to be a \emph{finite presentation} if both $X$ and $R$ are finite. If $X$ is a set, then we denote by $X^{-1}$ a set of the same cardinality as $X$ (considered an `inverse' set to $X$). We write $X^{*}$ for the set of finite words on $X \cup X^{-1}$, including the empty word $\emptyset$. If $g_{1}, \ldots, g_{n}$ are a collection of elements of a group $G$, then we write $\langle   g_{1}, \ldots, g_{n} \rangle$ for the subgroup in $G$ generated by these elements, and $\llangle g_{1}, \ldots, g_{n} \rrangle^{G}$ for the normal closure of these elements in $G$. The \emph{weight} of $G$, $\w(G)$, is the smallest $n$ such that $G=\llangle g_{1}, \ldots, g_{n} \rrangle^{G}$; to remove ambiguity, we set $\w(\{e\}):=0$. If $G$ is a group, then we write $G'$ for the derived subgroup of $G$, and $G^{\ab}:=G/G'$ for the abelianisation of $G$, where the commutator $[x,y]$ is taken to be $xyx^{-1}y^{-1}$; a group $G$ is said to be \emph{perfect} if $G^{\ab} \cong \{e\}$.

\subsection{Definition of finitely annihilated groups} 

We now formally define finitely annihilated groups, and hope that the reader will pick up the motivation for this by comparing it with that of a residually finite group as discussed in the introduction.

\begin{defn}\label{first defn}
Let $G$ be a group. An element $g \in G$ is said to be \textit{finitely annihilated} if there is a finite group $H_{g}$ and a homomorphism $\phi_{g}: G \to H_{g}$ such that $\phi_{g}(g)=e$ and $\image(\phi_{g}) \neq \{ e \}$. We say a non-trivial group $G$ is \textit{finitely annihilated} (F-A) if all its non-trivial elements are finitely annihilated. From hereon, we insist that the trivial group is not F-A.
\end{defn}

Note that we may drop the requirement that $\image(\phi_{g})$ is non-trivial, and instead insist that $\phi_{g}$ is a surjection to a non-trivial finite group $H$; this is clearly equivalent. The following equivalence is useful in the study of such groups.

\begin{lem}\label{second defn}
A group $G$ is F-A if and only if it is the union of all its proper, normal, finite index subgroups.
\end{lem}

We say that a normal subgroup $N\vartriangleleft G$ is \textit{maximal normal} if $G/N$ is simple.

\begin{prop}\label{max}
A group $G$ is F-A if and only if it is the union of all its maximal normal, proper, finite index subgroups.
\end{prop}

\begin{proof}
Suppose $G$ is F-A. Let $N\vartriangleleft G$ be proper and of finite index. Then the finite group $G/N$ is either simple (in which case $N$ is maximal normal in $G$), or has a maximal normal, proper subgroup whose preimage in $G$ is maximal normal, proper, and contains $N$. So we can replace each such $N$ by a maximal normal, proper, finite index subgroup containing it. The converse is immediate.
\end{proof}

From hereon we will usually find it convenient to use the covering by all maximal normal, proper, finite index subgroups when working with F-A groups.

\section{Properties of F-A groups}

We note some necessary and sufficient conditions for a group to be F-A.

\begin{prop}\label{alternative}
Let $G$ be a non-trivial group. Then $G$ is F-A if and only if neither of the following hold:
\\$1$. $G$ has weight $1$.
\\$2$. There is some $g \in G$ such that $G/ \llangle g \rrangle^{G}$ has no proper finite index subgroups.
\end{prop}

\begin{proof}
Suppose $G$ is F-A. Then, for each $g \in G$, $G/\llangle g \rrangle^{G}$ must have a non-trivial finite quotient, thus neither condition can hold. Conversely, if neither of the two conditions hold, then for any $g\in G$ we must have that $G/\llangle g \rrangle^{G}$ is non-trivial and has a finite quotient, so $G$ is F-A.
\end{proof}

\begin{prop}
Let $G$ be a finitely generated group which is neither F-A nor weight $1$. Then $G$ has an infinite simple quotient.
\end{prop}

\begin{proof}
If $\w(G) > 1$, then by proposition \ref{alternative} there exists $g \in G$ with $G/\llangle g \rrangle^{G}$ having no proper finite index subgroups. So either this is simple, or itself has a proper normal subgroup $H_{1}$ of infinite index. Then this quotient by $H_{1}$ is simple, or has a proper normal subgroup $H_{2}$ of infinite index. Continuing in this manner we get $H_{1}, H_{2}, \ldots$. Each $H_{i}$ has a preimage in $G$, call this $\tilde{H}_{i}$, all normal in $G$. We note that $\llangle g \rrangle^{G} \vartriangleleft \tilde{H}_{1} \vartriangleleft \tilde{H}_{2} \vartriangleleft \ldots$. But $G$ is finitely generated, so by Zorn's lemma the normal subgroup $H=\bigcup_{i \in \mathbb{N}} \tilde{H}_{i} $ is necessarily of infinite index, and moreover $G/H$ is simple.
\end{proof}

\begin{cor}\label{clarify}
Let $G$ be a finitely generated group with no infinite simple quotients. Then $G$ is F-A if and only if $\w(G) > 1$.
\end{cor}

Being F-A is independent of many other group properties. For example, there is no implication (in either direction) between being F-A and being any of finite, residually finite, or having solvable word problem. Moreover, being F-A is neither a quasi-isometry invariant, nor preserved by HNN extensions.

Looking at quotients is an important tool in understanding F-A groups, which we do now.

\begin{prop}\label{quot}
If $G$ surjects onto an F-A group, then $G$ must be F-A.
\end{prop}

\begin{proof}
$G/H$ is F-A, thus can be written as $G/H = \bigcup_{i \in I} N_{i}$, where each $N_{i}$ is proper, normal and of finite index. Let $\phi : G \twoheadrightarrow G/H$ be the quotient map. Then $G= \phi^{-1}(\bigcup_{i \in I} N_{i})=\bigcup_{i \in I} \phi^{-1}(N_{i})$. Moreover, each $\phi^{-1}(N_{i})$ is proper, normal, and of finite index in $G$, as $\phi$ is a surjection. Thus $G$ is F-A.
\end{proof}

That is, being F-A is preserved under reverse quotients. In particular, if $A$ is an F-A group, and $G$ any group, then $A *G$ and $A\times G$ will also be F-A.

The following gives a useful set of sufficient conditions which ensure that being finitely annihilated is preserved under quotients.

\begin{prop}\label{meta}
Let $G$ be a finitely generated F-A group, and $N\vartriangleleft G$. If $G=\bigcup_{i \in I} N_{i}$ is a covering by proper, normal, finite index subgroups, and $N$ is contained in every $N_{i}$, then $G/N$ is F-A.
\end{prop}

\begin{proof}
Take the quotient map $f: G \twoheadrightarrow G/N$. Then $f(N_{i})=N_{i}/N$ will be normal and of finite index in $G/N$, as $f$ is a surjection. But since $N \vartriangleleft N_{i}$ by hypothesis, we have that $(G/N)/(N_{i}/N) \cong G/N_{i}$, hence $f(N_{i})$ is also proper in $G/N$. So we have
\[
G/N = f(G)= f \left(  \bigcup_{i \in I} N_{i}\right )=\bigcup_{i \in I} f\left(N_{i}\right)
\]
and hence $\bigcup_{i \in I} f\left(N_{i}\right)$ is a proper, normal, finite index covering of $G/N$.
\end{proof}

We will make very frequent use of the above two results later on, when finding an alternate characterisation of F-A groups.

Clearly no non-trivial simple group is F-A. We now give a few explicit examples of other groups which are (or aren't) F-A.

\begin{lem}\label{Cp fa}
$G=C_{p}\times C_{p}$ is F-A for any prime $p$.
\end{lem}

\begin{proof}
$G=\bigcup_{g \in G} \left \langle g \right \rangle$ is a covering by proper, normal, finite index subgroups.
\end{proof}

We make very frequent use of the following lemma, which is immediate from the above lemma and proposition \ref{quot}.

\begin{lem}\label{cp}
Suppose a finitely generated group $G$ surjects onto $C_{p} \times C_{p}$ for some prime $p$. Then $G$ is F-A.
\end{lem}

\begin{lem}\label{classify free}
Let $X$ be a set. Then the free group on $X$, $F_{X}$, is F-A if and only if $|X|\geq 2$.
\end{lem}

\begin{proof}
If $|X|\geq 2$ then $F_{X}$ surjects onto $C_{2}\times C_{2}$, thus $F_{X}$ is F-A by lemma \ref{cp}. If $|X|\leq 1$, then $F_{X}$ is cyclic, hence not F-A by proposition \ref{alternative}.
\end{proof}

\begin{prop}\label{product}
A free product $G*S$ of a group $S$ having no proper normal finite index subgroups, and a group $G$ of weight $1$, is never F-A.
\end{prop}

\begin{proof}
Since $\w(G)=1$, there exists $g \in G$ such that $\llangle g \rrangle^{G}=G$. Hence $ \llangle g \rrangle^{G*S}=\llangle G \rrangle^{G*S}$, and so $(G*S)/\llangle g \rrangle^{G*S} = (G*S)/\llangle G \rrangle^{G*S} \cong S$.
Suppose $N$ is a proper, normal, finite index subgroup of $G *S$ containing $g$. Then $N$ contains $\llangle G \rrangle^{G*S}$, and so
\[
(G*S)/N \cong \left( (G*S)/\llangle G \rrangle^{G*S}\right) / \left(N/\llangle G \rrangle^{G*S}\right) \cong S/\left(N/\llangle G \rrangle^{G*S}\right)
\]
So $S$ has a proper, normal, finite index subgroup, which is impossible.
\end{proof}

\section{Embeddings, constructions and decidability}\label{emb sec}

We investigate the question of whether finitely presented F-A groups are algorithmically recognisable. The following two results, originally due to Higman, can be found in \cite{Serre}.

\begin{thm}[Higman {\cite[p.~9]{Serre}}]\label{hig}
Define the Higman group $H$ by
\[
 H := \left \langle a,b,c,d \ | \ aba^{-1}=b^{2}, bcb^{-1}=c^{2}, cdc^{-1}=d^{2}, dad^{-1}=a^{2} \right \rangle
\]
Let $\phi: H \to G$ be a homomorphism to some group $G$. Then $\phi(a), \phi(b),\phi(c),\phi(d)$ all have finite order if and only if they are all trivial.
\end{thm}

What interests us more is the following consequence:

\begin{lem}
The Higman group $H$ has no proper subgroup of finite index, and is torsion-free.
\end{lem}

\begin{proof}
Let $G\leq H$ be a finite index subgroup. Then $G$ has a subgroup $K$ which is normal in $H$ and of finite index. So we have the projection map $\phi_{K}: H \to H/K$, where $H/K$ is finite. The images of $a, b, c, d$ under this map all have finite order, and thus are trivial by theorem \ref{hig}. But $H/K$ is generated by these images, and is thus trivial. So $K=H$, and hence $G=H$. For the second part, note that the construction of $H$ (in \cite{Serre}) is via a finite sequence of amalgamated products and HNN extensions, beginning with free groups. Hence it is torsion-free. Hence by the torsion theorem for amalgamated products and HNN extensions (see \cite[Theorem 6.2]{Chiodo}), $H$ is torsion-free.
\end{proof}

The author wishes to thank Rishi Vyas for his contribution to the proof of the following lemma.

\begin{lem}\label{emb no fi}
Any finitely presented group $G$ can be uniformly embedded into a $2$-generator finitely presented group with no finite index subgroups.
\end{lem}

\begin{proof}
Let $G=\left \langle X |  R \right \rangle$ be any finitely presented group. We show that $G$ embeds into some finitely presented group which is not F-A. Take the free product of $G$ with the Higman group $H$, which has presentation
\[
G*H=\left \langle X,a,b,c,d \ | \ R, aba^{-1}=b^{2}, bcb^{-1}=c^{2}, cdc^{-1}=d^{2}, dad^{-1}=a^{2} \right \rangle
\]
Now form the $2$-generator finitely presented Adian-Rabin group $(G*H)(a)$ over the word $a$ (see \cite[Lemma 3.6]{Mille-92}). This group has no proper, finite index subgroups. For suppose so, then it would have a proper, normal, finite index subgroup $K$. Then $(G*H)(a)/K$ is finite, so by theorem \ref{hig}, the image of $a$ in this quotient is trivial. But by the Adian-Rabin relations, this means that the entire quotient is trivial. Hence $K=(G*H)(a)$. Finally, since $a \neq e$, we observe that $G\hookrightarrow G*H \hookrightarrow (G*H)(a)$, where $(G*H)(a)$ is a group without any proper subgroups of finite index, and thus not F-A.
\end{proof}

Combining this with proposition \ref{alternative}, we see that:

\begin{thm}\label{emb fa}
Every finitely presented group $G$ can be embedded (uniformly) into some $2$-generator finitely presented group for which no element is finitely annihilated, and hence is not F-A.
\end{thm}

We now give examples of groups which are neither F-A nor the normal closure of one element. $\mathbb{Q}$ is an obvious example, as it has no finite index subgroups. However, we construct a finitely presented example. To do this we require a partial result on the Kervaire conjecture, found as Theorem A in \cite{Kly-05}.

\begin{thm}[Klyachko {\cite[Theorem A]{Kly-05}}]\label{closure}
Let $G$ be torsion-free and non-trivial. Then the group $G * \mathbb{Z}$ has weight at least $2$.
\end{thm}

The following result, showing that being F-A is not equivalent to being weight 1, was inspired by a correspondence between the author and Mathieu Carette.

\begin{thm}\label{mathieu}
There is a $3$-generator finitely presentable group which is neither F-A nor the normal closure of any one element.
\end{thm}

\begin{proof}
Take the presentation $H$ for the Higman group from theorem \ref{hig}, and form the $2$-generator finitely presented Adian-Rabin group (see \cite[Lemma 3.6]{Mille-92}) with presentation $H(a)$ (where $a$ is one of the generators of $H$). An argument almost identical to the proof of lemma \ref{emb no fi} shows that $H(a)$ defines a finitely presented, infinite group with no finite index subgroups; that this group is torsion-free follows from the torsion-freeness of $H$ and \cite[Corollary 6.3]{Chiodo}. Now form $G:=\overline{H(a)}\ast \mathbb{Z}$. Then $G$ is neither the normal closure of any single element (by theorem \ref{closure}), nor F-A (by proposition \ref{product}).  
\end{proof}

We recall the definition of a Markov property.

\begin{defn}\label{markov defn}
An algebraic property of finitely presented groups $\rho$ is a \emph{Markov property} if there exist two finitely presented groups $G_{+}, G_{-}$ such that:
\\1. $G_{+}$ has the property $\rho$
\\2. $G_{-}$ does not have the property $\rho$, nor does it embed into any finitely presented group with the property $\rho$.
\end{defn}

It is a result by Adian and Rabin (see \cite[Theorem 3.3]{Mille-92}) that no Markov property is algorithmically recognisable amongst finitely presented groups; this is usually the way one shows a given property is algorithmically unrecognisable. However, as the following corollary to theorem \ref{emb fa} shows, this technique cannot be used here. We do not know if being F-A is an algorithmically recognisable group property amongst all finitely presented groups. 

\begin{cor}\label{mar}
Being F-A is neither a Markov property, nor a co-Markov property.
\end{cor}

\section{Classification and applications of F-A groups}\label{fa section}

In this section we describe a straightforward method to determine if a group is F-A, provided it lies in some particular collection of classes of groups.

The following result (Theorem 1 in \cite{BCK}) suggests that there is a strong relationship between being F-A and having non-cyclic abelianisation.

\begin{thm}[Brodie-Chamberlain-Kappe {\cite[Theorem 1]{BCK}}]\label{fin cover}
A group $G$ has a non-trivial finite covering by normal subgroups if and only if it has a quotient isomorphic to an elementary $p$-group of rank $2$ for some prime $p$.
\end{thm}

\begin{cor}\label{fin cover cor}
Let $G$ be a group that can be expressed as the union of finitely many proper, normal, finite index subgroups (and thus is F-A). Then $G$ has a quotient isomorphic to $C_{p} \times C_{p}$ for some prime $p$.
\end{cor}

We will eventually use the above result to characterise several classes of F-A groups in theorem \ref{list}. We begin with a characterisation of finite F-A groups. The following lemma is immediate from the structure theorem for finitely generated abelian groups (see \cite[\S I Theorem 2.1]{Hung}).

\begin{lem}\label{cyc surj}
Let $G$ be a finitely generated abelian group. Then $G$ is non-cyclic if and only if it surjects onto $C_{p} \times C_{p}$ for some prime $p$. 
\end{lem}

\begin{prop}\label{classify fintely many quotients}
Let $G$ be a finitely generated group with only finitely many distinct finite simple quotients. Then $G$ is F-A if and only if $G^{\ab}$ is non-cyclic.
\end{prop}

\begin{proof}
If $G^{ab}$ is non-cyclic then it surjects onto $C_{p} \times C_{p}$ for some prime $p$, and hence is F-A by lemma \ref{Cp fa}. Conversely, if $G$ is F-A, then it can be written as the union of all its maximal normal, proper, finite index subgroups (proposition \ref{max}). But as $G$ is finitely generated, it can only have finitely many subgroups of a given index. Since the index of these maximal normal, proper, finite index subgroups is bounded, there can only be finitely many of them. So by corollary \ref{fin cover cor}, $G \twoheadrightarrow C_{p}\times C_{p}$ for some prime $p$, and hence $G^{\ab}$ is non-cyclic.
\end{proof}

\begin{cor} \label{classify finite}
A finite group $G$ is F-A if and only if $G^{\ab}$ is non-cyclic.
\end{cor}

Similarly, we can now characterise finitely generated solvable F-A groups.

\begin{lem}\label{abelian chi}
Let $G$ be a finitely generated abelian group. Then $G$ is F-A if and only if it is non-cyclic. 
\end{lem}

\begin{proof}
If $G$ is cyclic then it is not F-A by theorem \ref{alternative}. Conversely, suppose $G$ is non-cyclic. Then by lemma \ref{cyc surj}, $G$ surjects onto $ C_{p} \times C_{p}$ for some prime $p$, and so is F-A by lemma \ref{cp}.
\end{proof}

\begin{lem}\label{pre nilpotent}
Let $G$ be a non-trivial finitely generated group whose finite simple quotients are all abelian (hence finite cyclic). Then $G$ is F-A if and only if $G^{\ab}$ is non-cyclic.
\end{lem}

\begin{proof}
If every finite simple quotient of $G$ is abelian, then $G'$ is a subgroup of every maximal normal finite index subgroup of $G$. So if $G$ is F-A, then by proposition \ref{meta}, $G^{\ab}$ is F-A, and hence non-cyclic by lemma \ref{abelian chi}. Conversely, if $G^{\ab}$ is non-cyclic, then by lemma \ref{abelian chi}, $G$ is F-A.
\end{proof}

\begin{prop}\label{classify nilpotent}
Let $G$ be a non-trivial finitely generated solvable group. Then $G$ is F-A if and only if $G^{\ab}$ is non-cyclic. 
\end{prop}

\begin{proof}
Any finite simple quotient of a solvable group will be abelian, so just apply lemma \ref{pre nilpotent}.
\end{proof}

The following lemma was observed in conjunction with Tharatorn Supasiti.

\begin{lem}\label{coprime}
Let $m,n$ be coprime positive integers. Then $\w(C_{m}*C_{n})=1$, and hence $C_{m}*C_{n}$ is not F-A.
\end{lem}

\begin{proof}
Let $P= \left \langle a,b\ |\ a^m, b^n \right \rangle$ be a finite presentation for $C_{m}*C_{n}$. Then $\overline{P} = \llangle  ab^{-1} \rrangle^{\overline{P}}$, as $m,n$ are coprime. So by proposition 
\ref{alternative}, $\overline{P}$ is not F-A.
\end{proof}

Combining this with proposition \ref{quot}, we get

\begin{prop} \label{classify two generator torsion}
Let $G$ be a two generator group, where the generators are torsion and of coprime order. Then $G$ is not F-A.
\end{prop}

We summarise our characterisation results so far:

\begin{thm}\label{list}
If $G$ is finitely generated and lies in at least one of the following classes, then $G$ is F-A if and only if $G^{\ab}$ (the abelianisation of $G$) is non-cyclic.
\\$1$. Free (proposition $\ref{classify free}$).
\\$2$. Solvable (proposition $\ref{classify nilpotent}$).
\\$3$. Having only finitely many distinct finite simple quotients (proposition $\ref{classify fintely many quotients}$). 
\\$4$. Two generator, with the generators having finite coprime order (proposition $\ref{classify two generator torsion}$).
\end{thm}

We now turn out attention to the `coverings' definition of F-A groups.

\begin{prop}
Define the following two conditions for groups.
\\$\cond 1$: $G$ is F-A if and only if $G^{\ab}$ is non-cyclic. 
\\$\cond 2$: $G$ can be covered by all its proper, normal, finite index subgroups if and only if there exists a finite subcover.
\\Then a set $S$ of finitely generated groups satisfies $\cond 1$ if and only if it satisfies $\cond 2$.
\end{prop}

\begin{proof}
Assume $S$ satisfies $\cond 1$. Given a finitely generated group $G$ which can be expressed as the union of all its proper, normal, finite index subgroups, we thus have that $G$ is F-A. So by hypothesis, $G^{\ab}$ is non-cyclic, so by lemma \ref{cyc surj}, $G$ surjects onto $C_{p} \times C_{p}$ for some prime $p$ (say via the map $f: G \twoheadrightarrow C_{p} \times C_{p}$). Take a finite covering $C_{p} \times C_{p} =\bigcup_{i=1}^{n}N_{i}$ by proper, normal, finite index subgroups. Then $G=f^{-1}(C_{p} \times C_{p})=f^{-1}(\bigcup_{i=1}^{n}N_{i})=\bigcup_{i=1}^{n}f^{-1}(N_{i})$ is a finite covering by proper, normal, finite index subgroups. The reverse direction is immediate. Thus $S$ satisfies $\cond 2$.
\\Now assume $S$ satisfies $\cond 2$. Let $G$ be a finitely generated group. If $G^{\ab}$ is non-cyclic then by lemma \ref{cp} we have that $G$ is F-A. Conversely, if $G$ is F-A, so by hypothesis $G$ has a finite covering by proper, normal, finite index subgroups. So by lemma \ref{cyc surj} $G$ surjects onto $C_{p}\times C_{p}$ for some prime $p$, and hence $G^{\ab}$ is non-cyclic. Thus $S$ satisfies $\cond 1$.
\end{proof}

We are very interested in the sets of finitely generated groups which satisfy $\cond 1$ (equivalently, $\cond 2$). Using the following theorem from \cite{Howie}, we show that not all sets of finitely generated groups satisfy this (pointed out to the author by Jack Button).

\begin{thm}[Howie, {\cite[Theorem 4.1]{Howie}}]\label{rep}
Let $w \in \{x,y,z\}^{*}$, and define $P:=\langle x,y,z\ |\ x^p, y^q, z^r,w \rangle$ where $p,q,r$ are distinct primes, and the exponent sums $\exp_{x}(w)$, $\exp_{y}(w)$, $\exp_{r}(w)$ (sums of powers of all instances of $x,y,z$ respectively in $w$) are coprime to $p,q,r$ respectively. Then there exists a representation $\rho: \overline{P} \to SO(3)$ with $\rho(x), \rho(y),\rho(z)$ having orders precisely $p,q,r$ respectively.
\end{thm}

\begin{prop}\label{cyclic ab fa}
Let $p,q,r$ be distinct primes. Then the group $\overline{K}\cong C_{p}*C_{q}*C_{r}$ with presentation $K:=\langle x,y,z\ |\ x^p, y^q, z^r \rangle$ is F-A, but $\overline{K}^{\ab}\cong C_{pqr}$ is cyclic. 
\end{prop}

\begin{proof}
This closely follows the proof of \cite[Corollary 4.2]{Howie}. Take a word $w\in \{x,y,z\}^{*}$ and define $P:=\langle x,y,z\ |\ x^p, y^q, z^r, w \rangle$, hence $\overline{P} \cong \overline{K}/\llangle w \rrangle^{\overline{K}}$ with associated quotient map $h: \overline{K} \twoheadrightarrow \overline{P}$. If $p$ divides $\exp_{x}(w)$, then $\overline{P}/\llangle y,z \rrangle^{\overline{P}} \cong C_{p}$, and $w$ is trivial in this  quotient $C_{p}$. A similar argument works when $q$ divides $\exp_{y}(w)$ or $r$ divides $\exp_{z}(w)$. Thus $w$ is a finitely annihilated element of $\overline{K}$ in any of these 3 cases. What remains is the case where $\exp_{x}(w),\exp_{y}(w),\exp_{r}(w)$ are each coprime to $p,q,r$ respectively. Now we may apply theorem \ref{rep} to show that there is a representation $\rho : \overline{P} \to SO(3)$ which preserves the orders of $x,y,z$. But then the non-trivial image of $\rho$ in $SO(3)$ will be residually finite, as it is a discrete subgroup of a linear group. So, since $|\rho(x)|=p>1$, then there is a finite group $H$ and a map $f: \image(\rho) \to H$ with $f(\rho(x))\neq e$. So the map $f \circ \rho \circ h : \overline{K} \to H$ annihilates $w$, and is a non-trivial map to a finite group. Thus $w$ is a finitely annihilated element in this last case, so $\overline{K}$ is F-A.
\end{proof}

From this we obtain a non-compactness result for coverings by proper normal finite index subgroups.

\begin{thm}\label{non-compact}
Let $p,q,r$ be distinct primes. Then the group $C_{p}*C_{q}*C_{r}$ is covered by its (infinitely many) proper normal finite index subgroups, but has no finite subcover by these.
\end{thm}

It seems natural to now ask the following question.

\begin{ques}\label{perfect fa}
Does there exist a finitely generated, perfect, F-A group?
\end{ques}

We suspect the answer to the above question to be no. At this point it makes sense to mention a closely related open problem in group theory, first posed by J. Wiegold as Question 5.52 in \cite{Notebook}: `Is every finitely generated perfect group necessarily weight $1$?' If the answer to this is yes, then the answer to question \ref{perfect fa} would be no (by proposition \ref{alternative}).

We now apply some of our results to prove various facts about groups.

\begin{thm}\label{tab}
Let $n>1$, and $G$ be a finitely generated group from a class given in theorem $\ref{list}$ such that $G$ has no infinite simple quotients. Then $\w(G) =n$ if and only if $\w(G^{\ab})=n$, and $\w(G)\leq 1$ if and only if $\w(G^{\ab}) \leq 1$. 
\end{thm}

\begin{proof}
We always have $\w(G^{\ab}) \leq \w(G)$ since $G^{\ab}=G/G'$ is a quotient of $G$. 
Now consider the case where $G$ is in a class that is preserved under taking quotients (i.e., class $2$ or $3$). If $\w(G^{\ab})\leq 1$ then, since $G$ belongs to a class from theorem \ref{list}, we have that $G$ is not F-A. But $G$ has no infinite simple quotients, so corollary \ref{clarify} shows that $\w(G)\leq 1$. If on the other hand $\w(G^{\ab})=n >1$, then take $n$ elements $g_{1}G', \ldots, g_{n}G'$ whose normal closure is all of $G/G'$. Setting $K:=G/\llangle g_{1}, \ldots, g_{n-1} \rrangle^{G}$ we see that $\w(K^{\ab})=1$, and hence $\w(K)= 1$ by what we have just shown. So $\w(G)\leq(n-1)+1=n$. But $\w(G^{\ab})=n$, so $\w(G) \geq n$. Combining these gives that $\w(G)=n=\w(G^{\ab})$.
\\ Finally, for the case where $G$ is in class $1$ or $4$, the inequality follows from elementary group theory.
\end{proof}

Seeing as the class of finite groups is listed in theorem \ref{list} we have the following immediate corollary, which is another way to resolve the Wiegold question for finite groups (already known in the literature, as a consequence of the main result by Kutzko in \cite{Kut}).

\begin{cor}\label{second fin}
Let $n>1$, and let $G$ be a finite or solvable group. Then $\w(G) =n$ if and only if $\w(G^{\ab})=n$, and $\w(G)\leq 1$ if and only if $\w(G^{\ab}) \leq 1$. 
\end{cor}

\begin{cor}\label{Wiegold fin}
Non-trivial finite (or solvable) perfect groups have weight $1$.
\end{cor}

\section{Generalisation: n-Finitely Annihilated}

We can generalise the definition of being F-A in a similar way to the definition of fully residually free groups from residually free groups (see \cite{Baum}). Almost all of our results for F-A groups carry over to our new definition in a natural way.

\begin{defn}[c.f.~definition \ref{first defn}]
Let $G$ be a group, and $n>0$. A collection of $n$ elements $g_{1}, \ldots, g_{n} \in G$ is said to be \textit{finitely annihilated} if there is a non-trivial finite group $H$ and a surjective homomorphism $\phi: G \twoheadrightarrow H$ such that $\phi(g_{i})=e$ for all $i=1, \ldots, n$. We say a non-trivial group $G$ is \textit{$n$-finitely annihilated} ($n$-F-A) if every collection of $n$ elements in $G$ is finitely annihilated. From hereon, we insist that the trivial group is not $n$-F-A for any $n$.
\end{defn}

Using the following definition of Brodie from \cite{Bro}, we give an equivalent interpretation of $n$-F-A  groups, which is very useful in the study of such groups.

\begin{defn}
An \emph{$n$-covering} of a group $G$ is a collection of subgroups $\{N_{i}\}_{i \in I}$ over an index set $I$ such that, for any set of $n$ elements $\{g_{1}, \ldots, g_{n}\} \subseteq G$, there is some $i \in I$ with $\{g_{1}, \ldots, g_{n}\} \subseteq N_{i}$.
\end{defn}

So an alternate equivalent interpretation of $n$-F-A is the following:

\begin{prop}[c.f.~proposition \ref{max}]
A group $G$ is $n$-F-A if and only if it has an $n$-covering by maximal normal, proper, finite index subgroups.
\end{prop}

It is then immediate that our definition of $n$-F-A groups really is a generalisation of F-A groups, in the following sense:

\begin{lem}
Let $G$ be $n$-F-A for some $n$. Then $G$ is $k$-F-A for every $k\leq n$.
\end{lem}

We now go over our results for F-A groups, and draw analogies to $n$-F-A groups. We state the most important of these, and provide proofs when it is not immediately obvious from the F-A case (where no proof is given, see the analogous case for F-A groups; the proof will be a straightforward adaptation).

\begin{prop}[c.f.~proposition \ref{alternative}]\label{n alternative}
Let $G$ be a non-trivial group. Then $G$ is $n$-F-A if and only if neither of the following hold:
\\$1$. $\w(G)\leq n$.
\\$2$. There is some $g_{1}, \ldots, g_{n} \in G$ such that $G/ \llangle g_{1}, \ldots, g_{n} \rrangle^{G}$ has no proper finite index subgroups.
\end{prop}

Just as in the F-A case, being $n$-F-A is preserved under reverse quotients.

\begin{prop}[c.f.~proposition \ref{quot}]\label{n quot}
Let G be a group for which there is some quotient $G/H$ which is $n$-F-A. Then G itself is $n$-F-A.
\end{prop}

Moreover, whenever we take a suitable quotient, being $n$-F-A is preserved.

\begin{prop}[c.f.~proposition \ref{meta}]\label{n meta}
Let $G$ be a finitely generated $n$-F-A group, and $N\vartriangleleft G$. If $G=\bigcup_{i \in I} N_{i}$ is a proper $n$-covering by normal finite index subgroups, and $N$ is contained in every $N_{i}$, then $G/N$ is $n$-F-A.
\end{prop}

\begin{proof}
Take the quotient map $f: G \twoheadrightarrow G/N$. Then $f(N_{i})=N_{i}/N$ will be normal, and finite index in $G/N$, as $f$ is a surjection. But since $N \vartriangleleft N_{i}$ by hypothesis, we have that $(G/N)/(N_{i}/N) \cong G/N_{i}$, hence $f(N_{i})$ is also proper in $G/N$. So we have
\[
G/N = f(G)= f \left(  \bigcup_{i \in I} N_{i}\right )=\bigcup_{i \in I} f\left(N_{i}\right)
\]
Moreover, if we take $g_{1}N, \ldots, g_{n}N \in G/N$, then there is some $j$ such that $g_{1}, \ldots, g_{n} \in N_{j}$ and hence $g_{1}N, \ldots, g_{n}N \in NN_{j}/N=N_{j}/N=f(N_{j})$. So $\bigcup_{i \in I} f\left(N_{i}\right)$ is an $n$-F-A covering of $G/N$.
\end{proof}

\begin{lem}[c.f.~lemma \ref{cyc surj}]
A finitely generated abelian group $G$ has weight $n$ if any only if it surjects onto an elementary abelian $p$-group of weight $n$ for some prime $p$; $C_{p}^{n}$.
\end{lem}

We now generalise the result by Brodie-Chamberlain-Kappe (theorem \ref{fin cover}) to the case of $n$-coverings for finitely generated groups. This has been proved for general groups by Brodie in \cite[Theorem 2.6]{Bro}. We provide a simple proof here for the finitely generated case, and use it to prove the $n$-F-A analogue of our characterisation in theorem \ref{list}.

\begin{thm}\label{gen Brodie}
A finitely generated group $G$ has a finite proper $n$-covering $\bigcup_{i=1}^{k} N_{i}$ by normal finite index subgroups if and only if $\w(G^{\ab}) \geq n+1$ (equivalently, if and only if $G$ surjects onto $C_{p}^{n+1}$ for some prime $p$).
\end{thm}

\begin{proof}
We need only prove the forward direction (the reverse is implied by proposition \ref{n quot}). We proceed by induction, the case $n=1$ is true by theorem \ref{fin cover}. Now suppose $G=\bigcup_{i=1}^{k} N_{i}$ exhibits the $(n+1)$-F-A property. Then it also exhibits the $n$-F-A property, so $\w(G^{\ab}) \geq n+1$. Take $g\in G$ with $gG'$ in a generating set of minimal size for $G^{\ab}$. As $G$ is $(n+1)$-F-A, then for all $g_{1}, \ldots, g_{n}$ there is an $N_{j}$ with $\{g, g_{1}, \ldots, g_{n} \} \subseteq N_{j}$. So $G/\llangle g \rrangle^{G}$ is $n$-F-A, and so has abelianisation of weight at least $ n+1$. But then $G$ has abelianisation of weight at least $ (n+1)+1$ (as we annihilated $gG'$ which was in a minimal generating set for $G^{\ab}$), so the induction is complete.
\end{proof}

By combining the above two results, we deduce the following analogue of lemma \ref{abelian chi}.

\begin{lem}[c.f.~lemma \ref{abelian chi}]\label{n abelian chi}
Let $G$ be a finitely generated abelian group. Then $G$ is $n$-F-A if and only if $\w(G) \geq n+1$.
\end{lem}

Many more of the results about F-A groups from sections \ref{emb sec} and \ref{fa section} can be generalised to $n$-F-A groups, following the proofs of their F-A counterparts. Here we state (without proof) the most useful of those; a characterisation of some $n$-F-A groups.

\begin{thm}[c.f.~theorem \ref{list}]\label{n list}
If $G$ is finitely generated and lies in at least one of the following classes, then $G$ is $n$-F-A if and only if $\w (G^{\ab}) \geq n+1$.
\\$1$. Free.
\\$2$. Solvable.
\\$3$. Having finitely many distinct finite simple quotients.
\end{thm}

\vspace{5pt}

\noindent \scriptsize{\textsc{Mathematics Department, University of Neuch\^{a}tel
\\Rue Emile-Argand 11, Neuch\^{a}tel, 2000, SWITZERLAND
\\maurice.chiodo@unine.ch}

\end{document}